\documentclass{article}
\usepackage{amsmath, amssymb, amsthm, amsfonts, amsxtra, latexsym, amscd,pb-diagram,  graphics,rotating}

\theoremstyle{definition}
\theoremstyle{plain}

\newtheorem{dn}{\bf Definition}

\newcommand{\C}{\mathcal{C}}

\newcommand{\A}{\mathcal{A}}
\newcommand{\tx}{\otimes }
\newcommand{\ts}{\oplus}
\newcommand{\ri}{\rightarrow }

\newcommand{\Lh}{\frak{L}}
\newcommand{\Rh}{\frak{R}}

\newcommand{\Fo}{\hat{F}}
\newcommand{\Lo}{\hat{L}}
\newcommand{\Ro}{\hat{R}}

\newcommand{\Go}{\hat{G}}

\newcommand{\La}{\Breve{L}}
\newcommand{\Ra}{\Breve{R}}

\newtheorem{thm}{\bf Theorem}[section]
 
\newtheorem{pro}[thm]{\bf Proposition}
\newtheorem{hq}[thm]{\bf Corrolary}

\theoremstyle{definition}

\begin{document}

\title{ \Large{\bf On the braiding of an Ann-category}}

\pagestyle{myheadings} 
\markboth{Nguyen Tien Quang-Dang Dinh Hanh}{On the braiding of an Ann-category}
\maketitle
\noindent {\MakeUppercase{Nguyen Tien Quang} and \MakeUppercase{Dang Dinh Hanh}}

\vspace{0.2cm}

\noindent{\small {\it Dept. of Mathematics, Hanoi National University of Education, Viet Nam\\
Email: nguyenquang272002@gmail.com\\
\indent \quad\ \   ddhanhdhsphn@gmail.com}}

\setcounter{tocdepth}{1}

\begin{abstract}   
A braided Ann-category $\A$ is an Ann-category $\A$ together with the braiding $c$ such that $(\A, \otimes, a, c, (I,l,r))$ is a braided tensor category, and $c$ is compatible with the distributivity constraints. The paper shows the dependence of the left (or right) distributivity constraint on other axioms. Hence, the paper shows the relation to the concepts of {\it distributivity category} due to M. L. Laplaza and {\it ring-like category} due to A. Frohlich and C.T.C Wall. 

The center construction of an almost strict Ann-category is an example of an unsymmetric braided Ann-category.
\end{abstract} 

\vspace{0.3cm}
 
\noindent {\small{\bf Mathematics Subject Classifications (2000):} 18D10.}

\vspace{0.3cm}
 
\noindent {\small{\bf Key words:} Braided Ann-category, braided tensor category, distributivity constraint, ring-like category.}

\vspace{0.2cm}

\section{Introduction}
The concept of a braided tensor category is introduced by Andr\'{e} Joyal and Ross Street [2] which is a necessary extension of a symmetric tensor category, since the center of a tensor category is  a braided tensor category but unsymmetric. The case of braided group categories was considered in the above work like a structure lift of the concept of group category [1].

In 1972, M.L. Laplaza introduced the concept of a distributivity category  [4]. After that, in [1], A. Frohlich and C.T.C Wall introduced the concept of ring-like category, with the axiomatics which is asserted to be simplier than the one of M.L. Laplaza.   These two concepts  are categorization of the concept of commutative rings, as well as a generalization of the category of modules on a commutative ring $R$. In order to have descriptions on structures, and cohomological classify them, N. T. Quang has introduced  the concept of Ann-categories [5], as a categorization of the concept of rings (unnecessarily commutative),  with requirements of invertibility of objects and morphisms of the ground category, similar to of group categories (see [1]). With these requiments, we can prove that each congruence class of Ann-categories is completely defined by the invariants: the ring $R$, the $R-$bimodule $M$ and an element in the  Mac Lane cohomology group  $H^{3}_{MacL}(R, M)$ (see [8]).

The concept of a braided Ann-category is a natural development of the concept of an Ann-category. The axioms of this concept presents natural reletions between the constraints respect to   $\oplus, \otimes.$ After that, we have proved the dependence of some axioms between the braiding and the distributivity constraints: thanks to  the braiding $c$, the distributivity constraint can be defined by only one side (right or left). Concurrently, the paper shows that  each symmetric Ann-category (which the braiding is symmetric) satisfies the axiomatics due to M. L.  Laplaza [4] and the axiomatics due to A. Frohlich and C.T.C Wall [1]. Moreover, as a corrolary, the paper shows a reduced system  axiomatics of the one due to M. L. Laplaza, and prove that this axiomatics is equivalent to the one due to A. Frohlich and C. T. C Wall.

In the last section, we present the {\it center construction} of an almost strict Ann-category, as an extension of the center construction of a tensor category due to AndrÐ Joyal and Ross Street. This center construction is an example of the concept of a braided, but unsymmetric Ann-category.

In this paper, we sometimes denote  $XY$ instead of  $X\tx Y$ of two objects  $X, Y$.

\section {The definition of a braided  Ann-category}

Firstly, let us recall the definition of a {\it braided tensor category} according to [2].

{\it A braiding for a tensor category $\mathcal V$ consists of a natural collection of isomorphisms \[c=c_{A,B}: A\tx B \stackrel{\sim}{\longrightarrow} B\tx A\]
in $\mathcal V$ such that the two diagrams} (B1) {\it and} (B2) {\it commute:}

\[\scriptsize\begin{diagram}
\node{(A\otimes B)\otimes C}\arrow{e,t}{c\tx id}\arrow{s,l}{a^{-1}}
\node{(B\otimes A)\otimes C}\arrow{e,t}{a^{-1}}
\node{B\otimes (A\otimes C)}\arrow{s,r}{id\tx c}\\
\node{A\otimes (B\otimes C)}\arrow{e,t}{c}
\node{(B\otimes C)\otimes A}\arrow{e,t}{a^{-1}}
\node{B\otimes (C\otimes A)}
\tag{B1}
\end{diagram}\]
\[\scriptsize\begin{diagram}
\node{A\otimes (B\otimes C)}\arrow{e,t}{id\tx c}\arrow{s,l}{a}
\node{A\otimes (C\otimes B)}\arrow{e,t}{a}
\node{(A\otimes C)\otimes B}\arrow{s,r}{c\tx id}\\
\node{(A\otimes B)\otimes C}\arrow{e,t}{c}
\node{C\otimes (A\otimes B)}\arrow{e,t}{a}
\node{(C\otimes A)\otimes B}
\tag{B2}
\end{diagram}\]

If $c$ is a braiding,  so is $c'$ given by  $c'_{A,B}=(c_{B,A})^{-1}$,  since (B2) is just obtained from (B1) by replacing $c$ with $c'$. A {\it symmetry} is a braiding $c$ which satisfies  $c'=c$. 

{\it A braided tensor category is a pair $(\mathcal V, c)$ consisting of a tensor category $\mathcal V$ and a braiding $c$.}

\begin{dn} $ [5, 7, 8]$
An Ann-category consists of:

i) A category $\A$ together with two bifunctors $\ts,\tx:\A\times\A\longrightarrow \A.$

ii) A fixed object $0\in \A$ together with naturality constraints $a^+,c^+,g,d$ such that $(\A,\ts,a^+,c^+,(0,g,d))$ is a symmetric categorical group.

iii) A fixed object $1\in\A$ together with naturality constraints $a,l,r$ such that $(\A,\tx,a,(1,l,r))$ is a monoidal $A$-category.

iv) Natural isomorphisms $\Lh,\Rh$
\begin{equation*}
\begin{array}{cccc}
\Lh_{A, X, Y} : &A\otimes (X\oplus Y) &\rightarrow& (A\otimes X)\oplus (A\otimes Y)\\
\Rh_{X, Y, A}:&(X\oplus Y)\otimes A &\rightarrow &(X\otimes A)\oplus (Y\otimes A)
\end{array}
\end{equation*}
such that the following conditions are satisfied:

(Ann-1) For each $A\in \A,$ the pairs $(L^A,\breve{L^A}),(R^A,\breve{R^A})$ defined by relations:
\begin{equation*}
\begin{array}{cccc}
L^A = A \otimes -&&R^A = - \otimes A\\
\La_{X, Y}^A = \Lh_{A, X, Y}&&\Ra_{X, Y}^A = \Rh_{X, Y, A}
\end{array}
\end{equation*}
 are $\ts$-functors which are compatible with $a^+$ and $c^+$.

(Ann-2) For all $ A,B,X,Y\in \A,$ the following diagrams:
\[\scriptsize\begin{diagram}
\node{(AB)(X\oplus Y)}\arrow{s,l}{\La^{AB}}
\node{A(B(X\oplus Y))}\arrow{e,t}{id_A \otimes \La^B}\arrow{w,t}{\ a_{A, B, X\oplus Y}}
\node{A(BX\oplus BY)}\arrow{s,r}{\La^A}\\
\node{(AB)X\oplus (AB)Y}
\node[2]{A(BX)\oplus A(BY)}\arrow[2]{w,t}{a_{A, B, X}\oplus a_{A, B, Y}}
\tag{1}\label{bd1}
\end{diagram}\]
\[\scriptsize\begin{diagram}
\node{(X\oplus Y)(BA)}\arrow{s,l}{\Ra^{BA}}\arrow{e,t}{a_{X\oplus Y, B, A}}
\node{((X\oplus Y)B)A}
\node{(XB\oplus YB)A}\arrow{w,t}{\Ra^B \otimes id_A}\arrow{s,r}{\Ra^A}\\
\node{X(BA)\oplus Y(BA)}\arrow[2]{e,t}{a_{X, B, A}\oplus a_{Y, B, A}}
\node[2]{(XB)A\oplus (YB)A}
\tag{2}\label{bd2}
\end{diagram}\]
\[\scriptsize\begin{diagram}
\node{(A(X\oplus Y))B}\arrow{s,l}{\La^A \otimes id_B}
\node{A((X\oplus Y)B)}\arrow{w,t}{a_{A, X\oplus Y, B}}\arrow{e,t}{id_A \otimes \Ra^B}
\node{A(XB\oplus YB)}\arrow{s,r}{\La^A}\\
\node{(AX\oplus AY)B}\arrow{e,t}{\Ra^B}
\node{(AX)B\oplus (AY)B}
\node{A(XB)\oplus A(YB)}\arrow{w,t}{a\  \oplus \ a}
\tag{3}\label{bd3}
\end{diagram}\]

\[\scriptsize
\begin{diagram}
\node{(A\oplus B)X\oplus (A\oplus B)Y}\arrow{s,l}{\Ra^X \oplus \Ra^Y}
\node{(A\oplus B)(X\oplus Y)}\arrow{w,t}{\La}\arrow{e,t}{\Ra}
\node{A(X\oplus Y)\oplus B(X\oplus Y)}\arrow{s,r}{\La^A \oplus \La^B}\\
\node{(AX\oplus BX)\oplus (AY\oplus BY)}\arrow[2]{e,t}{\text{v}}
\node[2]{(AX\oplus AY)\oplus (BX\oplus BY)}
\tag{4}\label{bd4}
\end{diagram}\]
commute, where $\text{v} = \text{v}_{_{U, V, Z, T}}: (U\oplus V)\oplus (Z\oplus T) \rightarrow (U\oplus Z)\oplus (V\oplus T)$ \ 
is the unique morphism built from $a^+,c^+,id$ in the symmetric categorical group $(\A,\ts).$

(Ann-3) For the unity object $1\in \A$ of the operation $\ts,$ the following diagrams:
\[\scriptsize\begin{diagram}
\node{1(X\oplus Y)}\arrow[2]{e,t}{\La^1}\arrow{se,b}{l_{X\oplus Y}}
\node[2]{1X \oplus 1Y}\arrow{sw,b}{l_X\oplus l_Y}
\node{(X\oplus Y)1}\arrow[2]{e,t}{\Ra^1}\arrow{se,b}{r_{_{X\oplus Y}}}
\node[2]{X1\oplus Y1}\arrow{sw,b}{r_{_X}\oplus r_{_Y}}\\
\node[2]{X\oplus Y}
\node{(5.1)}
\node[2]{X\oplus Y}
\node{(5.2)}
\end{diagram}\]
commute.
\end{dn}

 \begin{dn}  
A braided Ann-category  $\A$ is an Ann-category $\A$ together with a  braiding $c$ such that $(\A, \otimes, a, c, (I,l,r))$ is a braided tensor category, and $c$ makes the following diagram
 \[\scriptsize\begin{diagram}\label{bd5}
\node{A.(X\ts Y)}\arrow{e,t}{\La^A_{X,Y} }\arrow{s,l}{c}
\node{A.X\ts A.Y}\arrow{s,r}{c\ts c}\\
\node{(X\oplus Y).A}\arrow{e,t}{\Ra^A_{X,Y}}
\node{XA\ts YA}
\tag{6}
\end{diagram}\]
commutes and satisfies the condition:\   $c_{_{0, 0}}=id$.

A braided Ann-category  is called a symmetric Ann-category if the braiding $c$ is a symmetric. 
\end{dn}

\section{Some remarks on the axiomatics of a braided Ann-category}
 
In the axiomatics of a braided Ann-category, we can see that the diagram (6) allows us to   determine the right distributivity constraint $\Rh$ thanks to the left distributivity constraint and  the braiding $c$. So, we must consider the dependence or independence of the axioms which are  related to the right distributivity constraint $\Rh$, as well as the dependence of some other conditions in our axiomatics. 

Fristly, let us  recall a result which has known.

\begin{pro}[{Proposition 2[7]}] \label{md1}
In the axiomatics of an Ann-category $\A,$ the compatibility of the functors $(L^A, \breve{L}^A)$, $ (R^A, \breve{R}^A)$ with the commutativity constraint $c^+$ can be deduced from the other axioms.
\end{pro}
For the axiomatics of a braided Ann-category, we have the following result: 

\begin{pro}\label{md2}
In the axiomatics of a braided Ann-category $\A$, the compatibility of the functor $ (R^A, \breve{R}^A)$ [resp. $(L^A,\La^A)$] with the associativity constraint $a^+$ can be deduced from  the compatibility of the functor $(L^A, \La^A)$ [resp. $(R^A, \Ra^A)$] with the  constraint $a^+$ and the diagram (6). 
\end{pro}
\begin{proof}
To prove the compatibility of functor $(R^A, \Ra^A)$ with the associativity constraint $a^{+}$, let us  consider the following diagram:
\[\scriptsize\setlength\unitlength{0.5cm}
\begin{picture}(20,11)
\put(0, 10){$A((X\ts Y)\ts Z)$}
\put(0, 7){$((X\ts Y)\ts Z)A$}
\put(0, 4){$(X\ts (Y\ts Z))A$}
\put(0, 1){$A(X\ts (Y\ts Z))$}

\put(7.5, 10){$A(X\ts Y)\ts AZ$}
\put(7.5, 7){$(X\ts Y)A\ts ZA$}
\put(7.5, 4){$XA\ts (Y\ts Z)A$}
\put(7.5, 1){$AX\ts A(Y\ts Z)$}

\put(15, 10){$(AX\ts AY)\ts AZ$}
\put(15, 7){$(XA\ts YA)\ts ZA$}
\put(15, 4){$XA\ts (YA\ts ZA)$}
\put(15, 1){$AX\ts (AY\ts AZ)$}

\put(2.5, 9.7){\vector(0, -1){2.1}}\put(2.1, 8.5){c}
\put(2.5, 4.7){\vector(0, 1){2.1}}\put(2.7, 5.5){$a^+\tx id_a$}
\put(2.5, 1.5){\vector(0, 1){2.1}}\put(2.1, 2.5){c}

\put(9.5, 9.7){\vector(0, -1){2.1}}\put(7.9, 8.5){$c\ts c$}
\put(9.5, 1.5){\vector(0, 1){2.1}}\put(7.9, 2.5){$c\ts c$}

\put(16.5, 9.7){\vector(0, -1){2.1}}\put(16.7, 8.5){$(c\ts c)\ts c$}
\put(16.5, 4.6){\vector(0, 1){2.1}}\put(16.7, 5.5){$a^+$}
\put(16.5, 1.5){\vector(0, 1){2.1}}\put(16.7, 2.5){$c\ts (c\ts c)$}

\put(4.5, 10.1){\vector(1,0){2.5}}\put(5.5, 10.5){$\La$}
\put(11.8, 10.1){\vector(1,0){2.9}}\put(12.3, 10.5){$\La\ts id$}

\put(4.5, 7.1){\vector(1,0){2.5}}\put(5.5, 7.3){$\Ra$}
\put(11.8, 7.1){\vector(1,0){2.9}}\put(12.3, 7.3){$\Ra\ts id$}

\put(4.5, 4.1){\vector(1,0){2.5}}\put(5.5, 4.3){$\Ra$}
\put(11.8, 4.1){\vector(1,0){2.9}}\put(12.3, 4.3){$id\ts\Ra$}

\put(4.5, 1.1){\vector(1,0){2.5}}\put(5.5, 1.3){$\La$}
\put(11.8, 1.1){\vector(1,0){2.9}}\put(12.3, 1.3){$id\ts \La$}

\put(-2.2,10.1){\vector(1,0){1.9}}
\put(-2.2,10.1){\line(0,-1){9}}
\put(-0.2,1.1){\line(-1,0){2}}\put(-2,5.5){$id\tx a^+$}

\put(21,10.1){\vector(-1,0){1.3}}
\put(21,10.1){\line(0,-1){9}}
\put(19.5,1.1){\line(1,0){1.5}}\put(21.2, 5.5){$a^+$}

\put(1, 5,5){(I)}
\put(5.4, 8.5){(II)}
\put(12.4, 8.5){(III)}
\put(9, 5.5){(IV)}
\put(5.4, 2.5){(V)}
\put(12.4, 2.5){(VI)}
\put(18, 5.5){(VII)}
\end{picture}\]

In that above diagram, the region (I) commutes thanks to the naturality of $c$, the regions (II), (VIII), (IX) commute thanks to the diagram (6); the first component of the region (III) commutes thanks to the diagram (6),   the second component  one commutes thanks to the composition of morphisms, so the region (III) commutes; the first component of the region (VI) commutes thanks to the composition of morphisms, the second component  one commutes thanks to the diagram (6), so the region (VI) commutes; the region (VII) commutes thanks to the naturality of the isomorphism $a^+$, the perimeter commutes since $(L^A,\La^A)$ is compatible with associativity constraint $a^+$. Therefore, the region (IV) commutes, i.e., $(R^A, \Ra^A)$ is compatible with $a^+$.
\end{proof}

\begin{pro}
In a braided Ann-category, the commutativity of the diagrams (2) and (3) can be deduced from the other axioms.
\end{pro} 
\begin{proof}
To prove that the diagram (3) commutes, we consider the following diagram:
\[\scriptsize\setlength\unitlength{0.5cm}
\begin{picture}(21,17)
\put(2, 16){$A((X\ts Y)B)$}
\put(2, 13){$A(B(X\ts Y))$}
\put(2, 10){$(AB)(X\ts Y)$}
\put(2, 7){$(BA)(X\ts Y)$}
\put(2, 4){$B(A(X\ts Y))$}
\put(2, 1){$(A(X\ts Y))B$}

\put(9, 16){$A(XB\ts YB)$}
\put(9, 13){$A(BX\ts BY)$}
\put(9, 4){$B(AX\ts AY)$}
\put(9, 1){$(AX\ts AY)B$}

\put(16, 16){$A(XB)\ts A(YB)$}
\put(16, 13){$A(BX)\ts A(BY)$}
\put(16, 10){$(AB)X\ts (AB)Y$}
\put(16, 7){$(BA)X\ts (BA)Y$}
\put(16, 4){$B(AX)\ts B(AY)$}
\put(16, 1){$(AX)B \ts (AY)B$}

\put(3.7, 1.6){\vector(0,1){2}}\put(3, 2.5){$c$}
\put(3.7, 4.7){\vector(0,1){2}}\put(3, 5.5){$a$}
\put(3.7, 9.7){\vector(0,-1){2}}\put(1.6, 8.5){$c\tx id$}
\put(3.7, 12.7){\vector(0,-1){2}}\put(3, 11.5){$a$}
\put(3.7, 15.7){\vector(0,-1){2}}\put(1.6, 14.5){$id\tx c$}

\put(18, 1.6){\vector(0,1){2}}\put(18.3, 2.5){$c\ts c$}
\put(18, 4.7){\vector(0,1){2}}\put(18.3, 5.5){$a$}
\put(18, 9.7){\vector(0,-1){2}}\put(14.3, 8.5){$c\tx id\ts  c\tx id$}
\put(18, 12.7){\vector(0,-1){2}}\put(18.3, 11.5){$a$}
\put(18, 15.7){\vector(0,-1){2}}\put(15.8, 14.5){$id\tx c\ts\ \ id\tx c$}

\put(10.5, 1.6){\vector(0,1){2}}\put(10.7, 2.5){$c$}
\put(10.5, 15.7){\vector(0,-1){2}}\put(8, 14.5){$ id\tx (c \ \ts \ c)$}

\put(5.6,1.1){\vector(1,0){3}}\put(6.1,1.3){$\Lh\tx id$}
\put(12.5,1.1){\vector(1,0){3.1}}\put(13.8,1.3){$\Rh$}

\put(5.6,4.1){\vector(1,0){3}}\put(6.1,4.3){$id\tx \Lh$}
\put(12.5,4.1){\vector(1,0){3.1}}\put(13.8,4.3){$\Lh$}

\put(5.6,7.1){\vector(1,0){10.1}}\put(10.5,7.3){$ \Lh$}
\put(5.6,10.1){\vector(1,0){10.1}}\put(10.5,10.3){$ \Lh$}

\put(5.6,13.1){\vector(1,0){3}}\put(6.1,13.3){$id\tx \Lh$}
\put(12.5,13.1){\vector(1,0){3.1}}\put(13.8,13.3){$\Lh$}

\put(5.6,16.1){\vector(1,0){3}}\put(6.1,16.3){$id\tx \Rh$}
\put(12.5,16.1){\vector(1,0){3.1}}\put(13.8,16.3){$\Lh$}

\put(1.7,16.1){\line(-1,0){2}}
\put(-0.3,16,1){\line(0,-1){15}}\put(-1,8.5){$a$}
\put(-0.3,1.1){\vector(1,0){2}}

\put(20.4,16.1){\line(1,0){1.8}}
\put(22.2,16,1){\line(0,-1){15}}\put(20.6,8.5){$a\ts a$}
\put(22.2,1.1){\vector(-1,0){1.8}}

\put(6,14.5){(I)}
\put(13.2,14.5){(II)}
\put(10,11.5){(III)}
\put(10,8.5){(IV)}
\put(10,5.5){(V)}
\put(6,2.5){(VI)}
\put(13.2,2.5){(VII)}
\put(0.4,11.5){(VIII)}
\put(20,11.5){(IX)}
\end{picture}\]

In the above diagram, the regions (I) and (V) commute thanks to the diagram (6), the regions (II) and (IV) commute thanks to the naturality of $\Lh$, the regions (III) and (V) commute thanks to the diagram (1), the region (VI) commutes thanks to the naturality of $c$, the regions (VIII) and (IX) commute thanks to the diagram (B2). Therefore, the perimeter commutes, i.e., the diagram (3) commutes.

To prove that  the region (2) commutes, we consider the  following diagram:

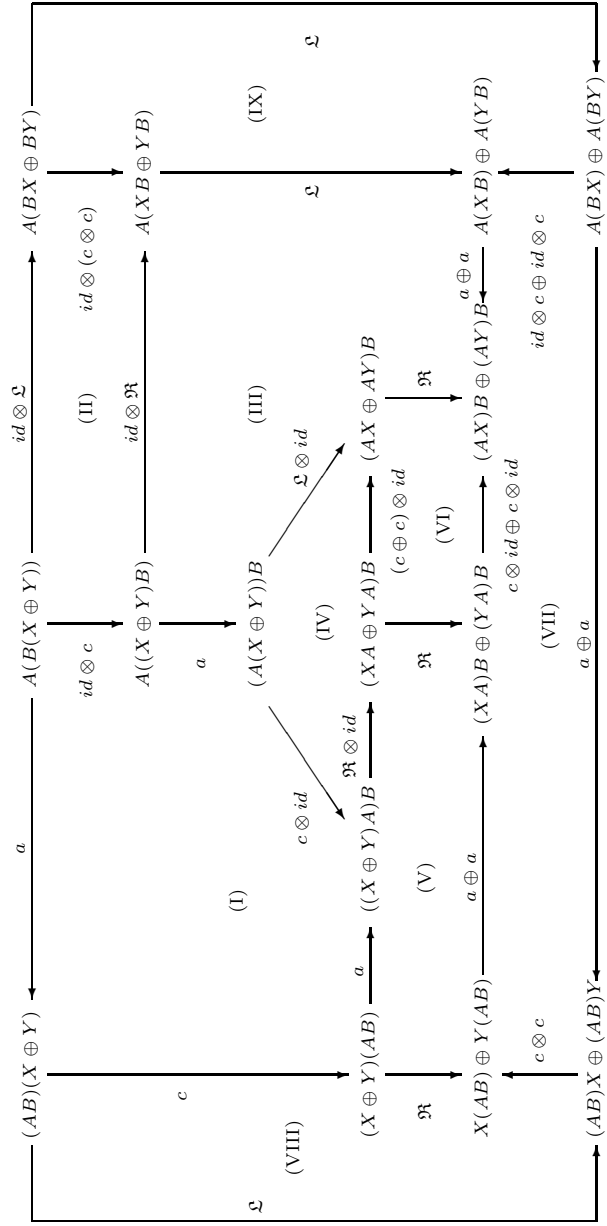
\[\begin{sideways}
\scriptsize\setlength\unitlength{0.5cm}
\begin{picture}(30,17)
\put(0, 16){$(AB)(X\ts Y)$}
\put(0, 7){$(X\ts Y)(AB)$}
\put(0, 4){$X(AB)\ts Y(AB)$}
\put(0, 1){$(AB)X\ts (AB)Y$}

\put(6, 7){$((X\ts Y)A)B$}

\put(12, 16){$A(B(X\ts Y))$}
\put(12, 13){$A((X\ts Y)B)$}
\put(12, 10){$(A(X\ts Y))B$}
\put(12, 7){$(XA\ts YA)B$}
\put(11, 4){$(XA)B\ts (YA)B$}

\put(18, 7){$(AX\ts AY)B$}
\put(18, 4){$(AX)B\ts (AY)B$}

\put(24, 16){$A(BX\ts BY)$}
\put(24, 13){$A(XB\ts YB)$}
\put(24, 4){$A(XB)\ts A(YB)$}
\put(24, 1){$A(BX)\ts A(BY)$}

\put(1.7, 1.6){\vector(0,1){2}}\put(2, 2.5){$c\tx c$}
\put(1.7, 6.7){\vector(0,-1){2}}\put(0.5, 5.5){$\Rh$}
\put(1.7, 15.7){\vector(0,-1){8}}\put(1, 12){$c$}

\put(25.7, 15.7){\vector(0,-1){2}}\put(22, 14.5){$id\tx (c\tx c)$}
\put(25.7, 12.7){\vector(0,-1){8}}\put(25, 8.5){$\Lh$}
\put(25.7, 1.7){\vector(0,1){2}}\put(21, 2.5){$id\tx c\ts id\tx c$}

\put(13.7, 15.7){\vector(0,-1){2}}\put(11.72, 14.5){$id\tx c$}
\put(13.7, 12.7){\vector(0,-1){2}}\put(12.5, 11.5){$a$}

\put(13.7, 6.7){\vector(0,-1){2}}\put(12.5, 5.5){$\Rh$}
\put(19.7, 6.7){\vector(0,-1){2}}\put(20, 5.5){$\Rh$}

\put(11.7,16.1){\vector(-1,0){8}}\put(7.58,16.3){$a$}
\put(15.6,16.1){\vector(1,0){8}}\put(18.5,16.3){$id\tx \Lh$}

\put(15.6,13.1){\vector(1,0){8}}\put(18.5,13.3){$id\tx \Rh$}

\put(3.6,7.1){\vector(1,0){2}}\put(4.3,7.3){$a$}
\put(9.6,7.1){\vector(1,0){2}}\put(9.6,7.5){$\Rh\tx id$}
\put(15.6,7.1){\vector(1,0){2}}\put(15,6.2){$(c\ts c)\tx id$}

\put(4.35,4.1){\vector(1,0){6.35}}\put(6.3,4.3){$a\ts a$}
\put(15.6,4.1){\vector(1,0){2}}\put(14.5,3.2){$c\tx id \ts c\tx id$}
\put(23.7,4.1){\vector(-1,0){1.45}}\put(22.3,4.5){$a\ts a$}

\put(23.7,1.1){\vector(-1,0){19.5}}\put(12.5,1.3){$a\ts a$}

\put(11.5, 9.8){\vector(-3,-2){3}}\put(7.8,8.8){$c\tx id$}
\put(15.5, 9.8){\vector(3,-2){3}}\put(17.3,8.8){$\Lh\tx id$}

\put(-0.2,16.1){\line(-1,0){2}}
\put(-2.2,16.1){\line(0,-1){15}}\put(-2, 10){$\Lh$}
\put(-2.2,1.1){\vector(1,0){2}}

\put(27.5,16.1){\line(1,0){2.7}}
\put(30.2,16.1){\line(0,-1){15}}\put(29, 8.5){$\Lh$}
\put(30.2,1.1){\vector(-1,0){1.8}}

\put(6,10.5){(I)}
\put(19,14.5){(II)}
\put(19,10){(III)}
\put(13.2,8.2){(IV)}
\put(6.5,5.5){(V)}
\put(15.8,5){(VI)}
\put(13,2.2){(VII)}
\put(-1,9){(VIII)}
\put(27,10){(IX)}
\end{picture}
\end{sideways}\]
\vspace{1cm}

In the above diagram, the regions (I) and (VII) commute since $(\A, \tx)$ is a braided tensor category; the regions (II), (IV) and (VIII) commute thanks to the diagram (6); the region (III) commutes thanks to Proposision 3.2;  the region (VI) commutes thanks to the naturality of $\Lh$; the perimeter commutes  thanks to the diagram (1). Therefore, the diagram (VI) commutes, i.e., the diagram (2) commutes.
\end{proof}

\begin{pro}
In the braided Ann-category  $\A$, the commutativity of the diagram (5.2) can be deduced from the diagrams (5.1), (6) and the compatibility of $c$ with the unitivity constraint $(I, l, r)$.
\end{pro}
\begin{proof} Consider the diagram:
\[\setlength\unitlength{0.5cm}%
\begin{picture}(13,7)
\put(0, 0){$1.(X\ts Y)$}
\put(10,0){$1.X\ts 1.Y$}

\put(5.7,3){$X\ts Y$}

\put(0, 6){$(X\ts Y).1$}
\put(10,6){$X.1\ts Y.1$}

\put(1.8, 5.8){\vector(0,-1){5.1}}
\put(11.5, 5.8){\vector(0,-1){5.1}}

\put(3.2, 0.2){\vector(1,0){6.6}}
\put(3.2, 6.2){\vector(1,0){6.6}}

\put(2, 5.8){\vector(2,-1){4}}
\put(11.3, 5.8){\vector(-2,-1){4}}

\put(2, 0.6){\vector(2,1){4}}
\put(11.3, 0.6){\vector(-2,1){4}}

\put(6.5, -0.6){$\La$}
\put(6.5, 6.4){$\Ra$}
\put(1.2, 3){$c$}
\put(12, 3){$c\ts c$}
\put(4.5, 1.2){$l$}
\put(7.5, 1.2){$l\ts l$}

\put(4.5, 4.7){$r$}
\put(7.5, 4.7){$r\ts r$}
\put(3,3){(I)}
\put(9,3){(IV)}
\put(6,1.2){(II)}
\put(5.5,4.7){(III)}
\end{picture}\]
\vspace{0.3cm}

In the above diagram, the regions (I) and (IV) commute thanks to the compatibility of $c$ with the unitivity constraint $(I,l,r)$; the region (II) commutes thanks to the diagram (5.1); the perimeter commutes thanks to the diagram (6). Therefore, the region (III) commutes, i.e., the diagram (5.2) commutes.  
\end{proof}

\noindent {\bf Remark 1.} According to Proposition 3.1 - 3.4, in the axiomatics of a braided Ann-category, we can omit the diagrams (2), (3), (5.2) and the compatibility of the functors $(R^A, \Ra^A)$ with the constraints  $a^+$, $ c^+$ of the operator $\oplus$.

  
\section{The object O}
In an Ann-category, the object O has important properties which is used to define the  $\Pi_{0}$-structure bimodule on $\Pi_1=Aut(0)$, where $\Pi_0$ is the ring of congruence classes of objects. 

\begin{pro}[{Proposition 1 [7]}]
In the Ann-category $\A$, there exist uniquely the isomorphisms:
$$\Lo^A:A\tx 0\ri A,\qquad \qquad \Ro^A:0\tx A\ri A$$
\noindent such that $(L^A, \La^A,\Lo^A),(R^A, \Ra^A, \Ro^A)$ are the functors which are compatible with the unitivity constraints of the operator $\oplus$ (also called $U$-functors).
\end{pro}

A part from some properties of the object 0 in an Ann-category (see [9]), we have the following proposition.

\begin{pro}
In a braided Ann-category, the following diagram 
\[\begin{diagram}
\node{X.0}\arrow[2]{e,t}{c}\arrow{se,l}{\Lo}
\node[2]{0.X}\arrow{sw,l}{\Ro}\\
\node[2]{0}
\tag{7}
\end{diagram}\]
commutes.
\end{pro}

\begin{proof}
Firstly, we can easily prove the following remark:

 {\it
Let $(F,\breve F), (G,\breve G): \C\ri \C'$ be  $\ts$-functors which are compatible with the unitivity constraints and   $\Fo: F(0)\rightarrow 0', \Go: G(0)\rightarrow 0'$ are isomorphisms, respectively. If   $\alpha: F \rightarrow G$ is an $\ts$-morphism such that $\alpha_0$ is an isomorphism, then the following diagram
\[\begin{diagram}
\node{F0}\arrow[2]{r,t}{\alpha_0}\arrow{se,b}{\hat F}\node[2]{G0}\arrow{sw,b}{\hat G}\\
\node[2]{0'}
\end{diagram}\]
\noindent commutes.}

We now apply the above remark to the  two functors
\begin{equation*}
F: A\tx -,\quad G: -\tx A.
\end{equation*}

The morphism $\alpha: F\rightarrow G$ is defined by
\begin{equation*}
\alpha_X=c_{A,X}: A\tx X\rightarrow X\tx A
\end{equation*}
  
\noindent is an $\oplus$-morphism thanks to the diagram (6). Since $\alpha_0=c_{A,0}$ is an isomorphism, arccording to  the above remark the diagram $(7)$ commutes.  
\end{proof}


\section{The relation between a symmetric Ann-category and a distributivity category due to  Miguel L. Laplaza}

Following, we will establish the relation between a symmetric Ann-category and a distributivity category due to [4], and deduce the coherence in a symmetric Ann-category [4].

A distributivity category $\A$ consists of:\\
i)\ \ \ Two bifunctors $\otimes, \oplus : \A\times \A\rightarrow \A$; \\
ii)\ \ Two fixed objects $O$ and $1$ of $\A$, called {\it null } and {\it unit objects};
\\
iii) Natural isomorphisms: 
\begin{eqnarray*}
a_{A,B,C}: A\otimes (B\otimes C)\rightarrow (A\otimes B)\otimes C, & \qquad c_{A, B}: A\otimes B\rightarrow B\otimes A\\
a^+_{A,B,C}: A\oplus (B\oplus C)\rightarrow (A\oplus B)\oplus C, &\qquad  c^+_{A,B}: A\oplus B\rightarrow B\oplus A\\
l_A: 1\otimes A\rightarrow A,& r_A: A\otimes 1\rightarrow A\\
g_A: O\oplus A\rightarrow A, & d_A: A\oplus O\rightarrow A\\
\Lo^A: A\otimes O \rightarrow O, & \Ro^A: O\otimes A \rightarrow O
\end{eqnarray*}
and natural isomorphisms:
\begin{eqnarray*}
\Lh_{A,B,C}: A\otimes(B\oplus C)\rightarrow (A\otimes B)\oplus (A\otimes C)\\
\Rh_{A,B,C}: (A\oplus B)\otimes C\rightarrow (A\otimes C)\oplus (B\otimes C)
\end{eqnarray*}
for all $A,B,C$ of $\A$.

These natural isomorphisms satisfy the  coherence conditions for $\{a,c,l,r\}$ and $\{a^+, c^+, g, d\}$; the functors $(L^A, \La^A), (R^A, \Ra^A)$ (similar as in the  concept of an Ann-category) are compatible with the associativity, commutativity constraints of the operator $\oplus$, and satisfy the commutative diagrams 
(1), (2), (3), (4), (5.1), (5.2), (6)  and 13 following diagrams:

\begin{equation}
\Lo^O=\Ro^O: O\otimes O\rightarrow O\tag{L1}
\end{equation}
\[\begin{diagram}
\node{O(A\oplus B)}\arrow{e,t}{\Lh}\arrow{s,l}{\Ro^{A\oplus B}}
\node{OA\oplus OB}\arrow{s,r}{\Ro^A\oplus \Ro^B}\\
\node{O}
\node{O\oplus O}\arrow{w,t}{g_{_O}}
\tag{L2}
\end{diagram}\]
\begin{equation}
d_{_O}(\Lo^A\oplus \Lo^B)\Rh_{A,B,O}=\Lo^{A\oplus B}: (A\oplus B)O\rightarrow O\tag{L3}
\end{equation}
\begin{equation}
r_1=g_{_{O}}: 1\otimes O\rightarrow O \tag{L4}
\end{equation}
\begin{equation}
l_1=d_{_O}: O\otimes 1\rightarrow O\tag{L5}
\end{equation}
\begin{equation}
\Lo^A=\Ro^A.c_{_{A,O}} : A\otimes O\rightarrow O\tag{L6}
\end{equation}
\[\begin{diagram}
\node{O(AB)}\arrow{e,t}{a_{O,A,B}}\arrow{s,l}{\Ro^{AB}}
\node{(OA)B}\arrow{s,r}{\Ro^A\otimes id_B}\\
\node{O}
\node{O\otimes B}\arrow{w,t}{\Ro^B}
\tag{L7}
\end{diagram}\]

\[\begin{diagram}
\node{A(OB)}\arrow[2]{e,t}{a_{A,O,B}}\arrow{s,l}{id_A\otimes \Ro^{B}}
\node[2]{(AO)B}\arrow{s,r}{\Lo^A\otimes id_B}\\
\node{AO}\arrow{e,t}{\Lo^A}
\node{O}
\node{OB}\arrow{w,t}{\Ro^B}
\tag{L8}
\end{diagram}\]

\begin{equation}
\Lo^{AB} \tx a_{A,O,B}=\Lo^A\tx (id_A\tx  \Lo^B): A(BO)\rightarrow O \tag{L9}
\end{equation}
\[\begin{diagram}
\node{A(O\oplus B)}\arrow{e,t}{\Lh_{A,O,B}}\arrow{s,l}{id_A\otimes g_{ B}}
\node{AO\oplus AB}\arrow{s,r}{\Lo^A\oplus id_{AB}}\\
\node{AB}
\node{O\oplus AB}\arrow{w,t}{g_{AB}}
\tag{L10}
\end{diagram}\]

\begin{equation}
g_{BA}(\Ro^A\oplus id_{BA})\Rh_{O,B,A}=g_{B}\otimes id_A: (O\oplus B)A\rightarrow BA\tag{L11}
\end{equation}
\begin{equation}
d_{AB}(id_{AB}\oplus \Lo^A)\Lh_{A,B,O}=id_A\otimes d^B: A(B\oplus O)\rightarrow AB\tag{L12}
\end{equation}
\begin{equation}
d_{AB}(id_{AB}\oplus \Ro^B)\Rh_{A,O,B}=d_A\otimes id_B: (A\oplus O)B\rightarrow AB\tag{L13}
\end{equation}

\begin{pro}
Each symmetric Ann-category is a distributivity category.
\end{pro}
\begin{proof} 
Let $\A$ be a symmetric Ann-category. From the definitions of a symmetric Ann-category and a distributivity category, we just must prove that  $\A$ satisfies the  conditions  (L1)-(L13).

 \noindent According to  Proposition 3.4 [9], the diagrams (L2), (L3) commute.

\noindent According to  Proposition 3.5 [9],  we have the equations (L4), (L5).

\noindent According to  Proposition 4.2, the diagram (L6) commutes.

 \noindent According to Proposition 3.3 [9], the diagrams (L7), (L8), (L9) commute.

\noindent According to  Proposition 3.1 [9], the diagrams (L10), (L11), (L12), (L13) commute.

\noindent From  the condition $c_{0,0}=id$ and  Proposition 4.2, we obtain $\Lo^o=\Ro^o$, i.e., the condition $(L1)$ is satisfied. So, each symmetric Ann-category is a distributivity category.
\end{proof}

In the above proof, only proof of the condition (L6) is related to the groupoid  property  of the  ground category, so we have

\begin{hq} In the axiomatics of a distributivity category, the axioms  (L1)-(L5), (L7)-(L13) are dependent.
\end{hq}

\begin{hq}
Let $\A$ be a category. $\A$ is a symmetric Ann-category iff $\A$ satisfies the  two  following conditions:\\
(i) $\A$ is a distributivity category;\\
(ii) Objects of category $(\A, \oplus)$ are all invertible and the ground category of $(\A, \oplus)$ is a groupoid.
\end{hq}

\begin{proof}
The necessary condition can be deduced from the definition of a symmetric Ann-category and Proposition 5.1.

The sufficient condition is obvious.
\end{proof}

We remark that the commutation of the diagrams in a symmetric Ann-category  $\A$ is independent of the invertibility of objects and morphisms in the category  $(\A, \oplus)$. Therefore, from the coherence theorem in a distributivity category [4], we have:

\begin{hq}
(Coherence theorem) In a symmetric Ann-category, each morphism built from  $a^+, c^+, g, d, a, c, l, r, \Lh, \Rh$ is just depent on the source and the target.
\end{hq}

The coherence theorem for a braided Ann-category is still an open problem.

\section{The relation between a  symmetric Ann-category and a ring-like category}

 Frohlich and C. T. C Wall have presented the concept of a {\it ring-like category} as a 
generalization of the category of modules on a commutative ring $R$ (see [1]).

 According to [1], a {\it ring-like} category consists of:\\
i)\ Two monoidal structures, given by the two functor: $\ts, \tx$;\\
ii) The distributivity  isomorphisms:
\[\Lh: A\otimes(B\oplus C)\rightarrow A\otimes B\oplus A\otimes C\]
\noindent such that the constraints   $a^+, c^+, a, c, \Lh$ satisfy the coherence  conditions, and the pairs $(L^A, \La^A)$ are the compatible with associativity, commutativity constraints of the operator $\oplus$,  satisfying the diagram (1) and the following diagram (8):
\[\scriptsize\setlength\unitlength{0.5cm}
\begin{picture}(18,10)
\put(1,9){$(A\ts B).(X\ts Y)$}
\put(1,7){$(X\ts Y).(A\ts B)$}
\put(0,5){$(X\ts Y).A\ts (X\ts Y).B$}
\put(0,3){$A.(X\ts Y)\ts B.(X\ts Y)$}
\put(-0.1,1){$(AX\ts AY)\ts (BX\ts BY)$}

\put(12.3,9){$(A\ts B)X\ts (A\ts B)Y$}
\put(12.3,7){$X(A\ts B)\ts Y(A\ts B)$}
\put(12,5){$(XA\ts XB)\ts (YA\ts YB)$}
\put(12,3){$(AX\ts BX)\ts (AY\ts BY)$}

\put(3.1, 7.5){\vector(0,1){1.3}}
\put(3.1, 6.8){\vector(0,-1){1.3}}
\put(3.1, 4.8){\vector(0,-1){1.3}}
\put(3.1, 2.8){\vector(0,-1){1.3}}

\put(15.3, 8.8){\vector(0,-1){1.3}}
\put(15.3, 6.8){\vector(0,-1){1.3}}
\put(15.3, 3.5){\vector(0,1){1.3}}
\put(6.5, 1.3){\vector(3,1){5.3}}

\put(5.5, 9.1){\vector(1,0){6.6}}

\put(3.3,8){$c$}
\put(3.3,6){$\Lh$}
\put(3.3,4){$c\ts c$}
\put(3.3,2){$\Lh\ts \Lh$}

\put(15.5,8){$c\ts c$}
\put(15.5,6){$\Lh\ts \Lh$}
\put(15.5,4){$(c\ts c)\ts (c\ts c)$}

\put(8.5, 9.3){$\Lh$}
\put(8.4, 2.2){$v$}
\put(12.3, 1) {\text{Diagram (8)}}
\end{picture}\]

\begin{pro}
Each symmetric Ann-category is a ring-like category.
\end{pro}

\begin{proof}
According to the definition of a symmetric Ann-category and the definition of a ring-like category, we just  prove that  in a symmetric Ann-category, the diagram (8) commutes. Indeed, we consider the following diagram:  
\[\scriptsize\setlength\unitlength{0.5cm}
\begin{picture}(15,14)
\put(0,13){$(X\ts Y)(A\ts B)$}
\put(0,10){$(A\ts B)(X\ts Y)$}
\put(-0.8,7){$(A\ts B)X\ts (A\ts B)Y$}
\put(-0.8,4){$X(A\ts B)\ts Y(A\ts B)$}

\put(12,13){$(X\ts Y)A\ts (X\ts Y)B$}
\put(12,10){$A(X\ts Y)\ts B(X\ts Y)$}
\put(11.65,7){$(AX\ts AY)\ts (BX\ts BY)$}
\put(11.65,4){$(AX\ts BX)\ts (AY\ts BY)$}
\put(11.65,1){$(XA\ts XB)\ts (YA\ts YB)$}

\put(2.1, 12.8){\vector(0,-1){2.2}}
\put(2.1, 9.8){\vector(0,-1){2.2}}
\put(2.1, 6.8){\vector(0,-1){2.2}}

\put(14.95, 12.8){\vector(0,-1){2.2}}
\put(14.95, 9.8){\vector(0,-1){2.2}}
\put(14.95, 6.8){\vector(0,-1){2.2}}
\put(14.95, 3.8){\vector(0,-1){2.2}}

\put(4.5, 13.1){\vector(1,0){7}}
\put(4.5, 10.1){\vector(1,0){7}}
\put(4.9, 6.7){\vector(3,-1){6.6}}
\put(4.9, 3.7){\vector(3,-1){6.6}}

\put(1.5, 11.5){$c$}
\put(1.5, 8.5){$\La$}
\put(0.5, 5.5){$c\ts c$}

\put(15.2, 11.5){$c\ts c$}
\put(15.2, 8.5){$\La\ts \La$}
\put(15.2, 5.5){$v$}
\put(15.2, 2.5){$(c\ts c)\ts (c\ts c)$}

\put(7.5, 13.2){$\La$}
\put(7.5, 10.2){$\Ra$}
\put(7.5, 6){$\Ra\ts \Ra$}
\put(6, 2){$\La\ts \La$}

\put(6, 0){\text{Diagram (9)}}
\put(7.3,11.5){(I)}
\put(7.3,8.3){(II)}
\put(7.2,4){(III)}
\end{picture}\]

In the above diagram, the region (I) commutes thanks to the diagram (6), the region (II) commutes thanks to the diagram (4); according to the diagram (6), each component of the region (III) commutes, so does the region (III). Then, the perimeter commutes, i.e., the diagram (8) commutes.  
\end{proof}

\begin{pro}
Let $\A$ be a category. Then, $\A$ is a symmetric Ann-category iff it satisfies the following  conditions:\\
i)  $\A$ is a ring-like category;\\
ii) All objects of the category $(\A, \oplus)$ are invertible and the ground of the  $(\A,\oplus)$ is a groupoid. 
\end{pro}

\begin{proof}  
The necessary condition can be implied from the definition of a symmetric Ann-category and  Proposition 6.1.

The sufficient condition can be implied from the definition of a ring-like category and the  following Proposition 6.3. 
\end{proof} 

In the section 2, we have proved: In a symmetric Ann-category, the diagrams $(2), (3),(4)$ and $(5.2)$ can be omitted. Now, we prove that we can omit the right distributivity constraint.

\begin{pro}
 Let $\A$ be a category and two bi-fucntors $\oplus, \otimes: \A\times A\rightarrow \A$. Then, $\A$ is a symmetric Ann-category iff the following  conditions are satisfied:\\
(i)  There is an object $O\in \A$ and naturality isomorphisms  $a^+, c^+, g, d$ such that  $(\A, \oplus, a^+, c^+, (O, g, d))$ is a symmetric categorical group.\\
(ii)   There is an object $1\in \A$ together with naturality constraints  $a, c, l, r$ such that   $(\A, \otimes, a, c, (1, l, r))$ is a  symmetric monoidal category.\\
(iii)  There is a naturality isomorphism $\Lh: A\otimes (X\oplus Y)\rightarrow (A\otimes X)\oplus(A\otimes Y)$
 such that $(L^A=A\otimes -, \La^A_{X, Y}=\Lh_{A, X, Y})$ are $\oplus$-functors which are compatible with $a^+$ and $c^+$, and the diagrams (1), (8), (5.1) commute.
\end{pro}

\begin{proof}
The necessary condition is obviously.\\
The sufficient condition: Let $\A$ be a category satisfying all conditions in the above proposition. The right distributivity constraint in $\A$ is defined by
\[\Rh: (X\ts Y)\tx A\rightarrow X\tx A\ts Y\tx A\]  
\noindent thanks to the commutative diagram (6). Then, according to Propositions  $3.2-3.4$, to prove that  $\A$ is a symmetric Ann-category, it remains to prove that the diagram (4) commutes.
 
\noindent We consider the diagram (9). In this diagram, the regions (I) and (III) commute thanks to the definition of the distributivity constraint $\Rh$, the perimeter commutes thanks to the diagram (8). Hence, the region (II) of the diagram (9) commutes, i.e., the diagram (4) commutes. The proposition has been poved.
\end{proof}
A. Frohlich and C. T. C Wall commented that the axiomatics of a distributivity category due to  M. L. Laplaza [4] is complicated (see [1]). So, they presented the definition of a  {\it ring-like category} in order to present a reduced version. However, in  [1], authors have not show the relation between these two definitions.

 As a corollary of the Propositions 5.3, 6.2, we have the  following result:
\begin{pro}
 The two concepts of: a distributivity category and a ring-like category are equivalent. 
\end{pro}

\begin{proof}

Let $\A$ be a distributivity category. From the definition of a distributivity category and a ring-like category, it remains to show that in a ring-like category, the diagram (8) commutes. 

Consider the diagram (9). Since, in a distributivity category, the diagrams (4) and (6) commute, so from the proof of  Proposition 6.1, we imply that the perimeter of the diagram (9) commutes, i.e., the diagram (8) commutes.

Inversely, let $\A$ be a ring-like category. In $\A$, we put
\[\Rh: (X\ts Y)A\ri XA\ts YA.\]
It is a naturality constraint which is defined by the commutative diagram $(6)$. Then, we prove similarly as Propositions 3.1, 3.2, 3.3, in a ring-like category, we deduced that the diagrams  $(2)$, $(3)$ commute and   we deduced that the functors  $(R^A, \Ra^A)$ are compatible with the associativity, commutativity of the operator of the $\oplus$.

Finally, it remains that the diagram  $(4)$ commutes. We consider the diagram (9). Since, in a ring-like category, the diagrams  $(6)$ and (8)  commute, according to the proof of Proposition 6.1, we obtain the region (II) of the diagram (9) commutes, i.e., the diagram (4) commutes. So, each ring-like category is a distributivity category.
\end{proof}


\section{The center of an almost strict Ann-category}

Andr\'{e} Joyal and Ross Street [2] have built a  braided tensor $L_{\mathcal V}$ of a strict tensor category $(\mathcal{V}, \otimes)$.  C. Kassel [3] has presented  one for an arbitrary tensor category.

According to  [3], the center $L_{\mathcal V}$ of the tensor category $\mathcal V$ is a category whose objects are pairs $(A,u)$, where $A\in Ob(\mathcal V)$ and
\[u_{X}: A\otimes X\longrightarrow X\otimes A\]
is a naturality transformation satisfying following commutative diagrams:
\[\scriptsize\begin{diagram}
\node{A\otimes I}\arrow[2]{e,t}{u_I}\arrow{se,b}{r}
\node[2]{I\otimes A}\arrow{sw,b}{l}\\
\node[2]{A} 
\end{diagram}\]
\[\scriptsize
\begin{diagram}
\node{(A\otimes X)\otimes Y}\arrow{e,t}{a^{-1}}\arrow{s,l}{u_X\otimes id_Y}
\node{A\otimes (X\otimes Y)}\arrow{e,t}{u_{X\tx Y}}
\node{(X\otimes Y)\otimes A}\arrow{s,r}{a^{-1}}\\
\node{(X\otimes A)\otimes Y}\arrow{e,t}{a^{-1}}
\node{X\otimes (A\otimes Y)}\arrow{e,t}{id_X\tx u_Y}
\node{X\otimes (Y\otimes A)}
\end{diagram}\]

A morphism $f: (A,u)\rightarrow (B,m)$ in $L_{\mathcal V}$ is a morphism $f: A\rightarrow B$ satisfies  the condition
\begin{equation}
m_X\circ(f\otimes id)=(id\otimes f)\circ u_Y
\tag{$10$}
\end{equation}
for all   $X\in {\mathcal V}$.

The tensor product of two objects in $L_{\mathcal V}$ is defined:
\[(A,u)\times (B,m)=(A\otimes B, u\times m)\]
\noindent in which  $u\times m$ is a morphism defined by the following commutative diagram:

\[\begin{diagram}
\node{(A\otimes B)\otimes X}\arrow{s,l}{(u\times m)_X}
\node{A\otimes (B\otimes X)}\arrow{e,t}{id\otimes m_X}\arrow{w,t}{a}
\node{A\otimes(X\otimes B)}\arrow{s,r}{a}\\
\node{X\otimes (A\otimes B)}\arrow{e,t}{a}
\node{(X\otimes A)\otimes B}
\node{(A\otimes X)\otimes B}\arrow{w,t}{u_X\otimes id}
\tag{11}
\end{diagram}\]
Then $L_{\mathcal V}$ is a braided tensor category with the braiding defined by
\[c_{_{(A,u), (B,m)}}=u_B: (A,u)\times (B,m)\rightarrow (B,m)\times (A,u).\]

Now, we will build  {\it the tencer} of an  Ann-category $\mathcal A$ and the main result of this section is Theorem 7.3. Firstly, let us  recall that each Ann-category is Ann- equivalent to the  reduced  one of the type $(R,M)$. Moreover, each  Ann-category  of the type  $(R,M)$ is equivalent to an   almost strict  Ann-category on it (see [6]), and in this category,  the family of morphisms  $i_X:X\ts X'\ri O$ is identity (an Ann-category is called {\it almost strict} if all its natural constraints, except for the commutativity constraint of the operation $\oplus$ and the left distributivity constraint, are identities). So, in this section, we always assume that  $\mathcal A$ is an almost strict  Ann-category and the morphisms $i_X:X\ts X'\ri O$ are identity.

\begin{dn}
The center of an Ann-category $\A$, denoted by  $C_{\mathcal A}$, is a category in which objects are pairs $(A,u)$, with $A\in Ob(\A)$ and
\begin{equation*}
u_{X}: A\otimes X \longrightarrow X\otimes A
\end{equation*} 
is a natural transformation satisfying the three conditions
\begin{eqnarray*}
(C1):\qquad \ \quad\ \  u_I&=&id\\
(C2):\qquad  \ u_{X\tx Y}&=&(id_X\tx u_Y)\circ (u_X\tx id_Y)\\
(C3):\qquad  \ u_{X\ts Y}&=&(u_X\ts u_Y)\circ \La_{A,X,Y}.
\end{eqnarray*}
\noindent The morphism $f: (A,u)\rightarrow (B,m)$ of $C_{\mathcal A}$ is a morphism $f: A\rightarrow B$ of $\mathcal A$ satisfying the condition $(10)$. 
\end{dn}

\begin{pro}
The center of an almost strict Ann-category $\A$ is a symmetric categorical group with the sum of two objects are defined by
\begin{equation*}
 (A,u)+(B,m)=(A\ts B, u+m)
\end{equation*}
in which: 
\begin{equation*}
(u+m)_X=\La_{X,A,B}^{-1}\circ (u_X\ts m_X)
\end{equation*}
and the sum of two morphisms of\  $C_{\A}$ is the  sum of morphisms in $\A$.
\end{pro}

\begin{proof}
Firstly, we prove that for two objects $(A,u), \ (B,m)$ of $C_{\A}$ then\linebreak $(A\oplus B, u+m)$ defined above is an object of $C_{\A}$. Indeed, since $u_1=id,$\linebreak  $ m_1=id, \La_{1,A,B}=id$, we have $(u+m)_1=id$. On the other hand, we have
\begin{eqnarray*}
&&(u+m)_{XY}\\
&=&\La_{XY,A,B}^{-1}\circ (u_{XY}\ts m_{XY})\\
&=&(id_X\tx \La_{Y,A,B})^{-1}\circ \La_{X,YA,YB}^{-1}\circ (id_X\tx u_Y\ts id_X\tx m_Y)\\
&& \circ (u_X\tx id_Y\ts m_X\tx id_Y) \\
&& \quad \text{(according to the definition of an Ann-category) }\\
&=&(id_X\tx \La_{Y,A,B})^{-1}\circ (id_X\tx(u_Y\ts m_Y))\circ \La_{X,AY,BY}^{-1}\circ (u_X\tx id_Y\ts m_X\tx id_Y)\\
&&\quad (\text{thanks to the naturality of}\ \La)\\
&=&(id_X\tx \La_{Y,A,B})^{-1}\circ (id_X\tx(u_Y\ts m_Y))\circ \La_{X,AY,BY}^{-1}\circ ((u_X\ts m_X)\tx id_Y)\\
&=&(id_X\tx (u+m)_Y)\circ \La_{X,AY,BY}^{-1}\circ ((u_X\ts m_X)\ts id_Y)\\
&&\quad (\text{thanks to the definition of  $u + m$})\\
&=&(id_X\tx (u+m)_Y)\circ (\La_{X,A,B}\tx id_Y)^{-1}\circ ((u_X\ts m_X)\tx id_Y)\\
&=&(id_X\tx (u+m)_Y)\circ ((u+m)_X\tx id_Y)  \\
&&\ \  (\text{thanks to the definition of  $u+m$})
\end{eqnarray*}
So, $u+m$ satisfies the condition (C2).\\
Next, we verify that the morphism $u+m$ satisfies the condition (C3). 
\begin{eqnarray*}
&&(u+m)_{X\ts Y}\\
&=&\La_{X\ts Y,A,B}^{-1}\circ (u_{X\ts Y}\ts m_{X\ts Y})\\
&=&\La_{X\ts Y,A,B}^{-1}\circ [(u_X\ts m_X)\ts (u_Y\ts m_Y)](\La_{A,X,Y}\circ \La_{B,X,Y}) \\
&& \quad \text{(since $u,m$ satisfy the condition (C3))}\\
&=&\La_{X\ts Y,A,B}^{-1}\circ v_{XA,YA,XB,YB}^{-1}\circ [(u_X\ts m_X)\ts (u_Y\ts m_Y)]\circ v_{AX, AY, BX, BY}\\
&&\circ [\La_{A,X,Y}\circ \La_{B,X,Y} ]\\
&& \quad (\text{thanks to the naturality of the isomorphism $v$})\\
&=&(\La_{X,A,B}\ts \La_{Y,A,B})^{-1}\circ ((u_X\ts m_X)\ts (u_Y\ts m_Y))\circ  \La_{A\ts B,X,Y}\\
&&\quad (\text{thanks to the definition of an Ann-category})\\
&=&((u+m)_X\ts (u+m)_Y)\circ \La_{A\ts B, X,Y}\\
&&\quad (\text{thanks to the definition of}\ u+m)
\end{eqnarray*}
So $u+m$ satisfies the condition (C3), i.e., $(A\oplus B, u+m)$ is an object of $C_{\A}$.

Now, we assume that  $f: (A,u)\ri (B,m), g: (A',u')\ri (B', m')$ are morphisms of  $C_{\A}$. From the definition of the sum of two objects, the sum of morphisms in  $C_{\A}$, and the naturality of  the isomorphism $\Lh$, and $\Rh=id$, we can verify that $f+g=f\oplus g$ is a morphism of the category  $C_{\A}$. Furthermore, we can verify  that
\begin{equation*}
id: ((A,u)+(B,m))+(C,w)\rightarrow (A,u)+((B,m)+(C,w))
\end{equation*}
\noindent is the associativity constraint of the category $C_{\A}$,  $((O, \theta_X=\Lo^{-1}_X),id,id)$  is the unitivity constraint of $C_{\A}$, and
\begin{equation*}
c^+_{A,B}: (A,u)+(B,m)\ri (B,m)+(A,u)
\end{equation*}
\noindent is the commutativity constraint of $C_{\A}$. Finally, we prove that each object of $C_{\A}$ is invertible.

Let  $(A,u)$ be an object of the category  $C_{\A}$. According to the Ann-category $\A$, there exists $A'\in Ob(\A)$ such that $A\ts A'=O$. We define a natural transformation \ $u'_X: A'X\ri XA'$\ as follows:
\begin{equation*}
u_X+u'_X=\La_{X,A,A'}\circ \theta_X
\end{equation*}
Then, we have
\begin{equation*}
(u+u')_X=\La_{X,A,A'}^{-1}\circ (u_X+u'_X)=\theta_X.
\end{equation*}
 
 We can easily verify that $u'$ satisfies the  conditions  (C1) and (C2). Now, we will verify that the morphism  $u'$ satisfies the condition (C3).
\begin{eqnarray*}
&& u_{X\ts Y}+u'_{X\ts Y}\\
&=&\La_{X\ts Y, A, A'}\circ \theta_{X\ts Y}\quad \hspace{3cm} (\text{thanks to the definition of }\ u')\\
&=&v_{XA,YA,XA',YA'}^{-1}\circ(\La_{X,A,A'}\ts \La_{Y,A,A'})\circ (\theta_{X\ts Y})\\
&& (\text{according to the definition of an Ann-category  })\\
&=&v_{XA,YA,XA',YA'}^{-1}\circ(\La_{X,A,A'}\ts \La_{Y,A,A'})\circ 
(\theta_{X}\ts \theta_{Y})\quad (\text{since}\ A\ts A'=O)\\
&=& v_{XA,YA,XA',YA'}^{-1}\circ(\La_{X,A,A'}\ts \La_{Y,A,A'})\circ ((u_X\ts u'_X)\ts (u_Y\ts u'_Y))\\
&&(\text{thanks to the  definition of} \ u' \ \text{and }\ u+u'=\theta)\\
&=&((u_X\ts u_Y)\ts (u'_X\ts u'_Y))\circ v_{AX,A'X, AY,A'Y}\\
&& (\text{thanks to the naturality of} \ v)\\
&=&((u_X\ts u_Y)\ts (u'_X\ts u'_Y))\circ (\La_{A,X,Y}\ts \La_{A',X,Y})\\
&&
(\text{thanks to the definition of an Ann-category}).
\end{eqnarray*}
Together with the equation  $u_{X\ts Y}=(u_X\ts u_Y)\circ\La_{A,X,Y}$, we have:
\begin{equation*}
u'_{X\ts Y}=(u'_X\ts u'_Y)\circ\La_{A',X,Y}
\end{equation*}
So $(A',u')$ is an object of the category  $C_{\A}$ and it is the invertible object of $(A,u)$ respect to the operator  +.
\end{proof}


\begin{pro}
The $C_{\A}$ is a braided tensor category where the tensor product of two objects is defined by
\begin{equation*}
(A,u)\times (B,m)=(A\otimes B, u\times m)
\end{equation*}
\noindent in which\ $u\times m$ is the morphism given by the diagram (11), and the tensor product of two morphisms in $C_{\A}$ is indeed the tensor product in $\A$.
\end{pro}

\begin{proof}
Let  $(A,u), (B,m)$ be objects  of  $C_{\A}$. According to  A. Joyal and R. Street [2], $u\times m$ satisfies the  two conditions $(C1)$ and $(C2)$. On the other hand, $u\times m$ satisfies the condition (C3), since:
\begin{equation*}
\begin{split}
&\quad \ \ (u\times m)_{X\ts Y}=(u_{X\ts Y}\tx id_B)\circ (id_A\tx m_{X\ts Y})\\
&=((u_X\ts u_Y)\tx id_B)\circ (\La_{A,X,Y}\tx id_B)\circ (id_A\tx m_{X\ts Y})\\
&\qquad (\text{since}\  u\ \text{ satisfies the condition (C3)} )\\
&= ((u_X\ts u_Y)\tx id_B)\circ (\La_{A,X,Y}\tx id_B)\circ (id_A\tx (m_X\ts m_Y))\circ(id_A\tx \La_{B,X,Y})\\
&\qquad (\text{since}\  m\ \text{satisfies the condition (C3)})\\
&=(u_X\tx id_B\ts u_Y\tx id_B)\circ (\La_{A,X,Y}\tx id_B)\circ (id_A\tx(m_X\ts m_Y))\circ (id_A\tx \La_{B,X,Y})\\
&\qquad (\text{thanks to the definition of the Ann-category $\A$ and}\ \Ra=id)\\
&= (u_X\tx id_B\ts u_Y\tx id_B)\circ\La_{A,XB,YB}\circ (id_A\tx(m_X\ts m_Y))\circ (id_A\tx \La_{B,X,Y}) \\
&\qquad (\text{according to the definition of an Ann-category})\\
&=((u_X\tx id_B)\ts (u_Y\ts id_B)) \circ(id_A\tx m_A\ts id_A\tx m_Y)\circ \La_{A,BX,BY}\\
&\ \ \ \  \circ(id_A\tx \La_{B,X,Y})\qquad \hspace{4.5cm}(\text{since}\  \La\ \text{ is natural})\\
&=((u\times m)_X\ts (u\times m)_Y)\circ \La_{A,BX,BY}\circ(id_A\tx \La_{B,X,Y})\\
&\qquad (\text{thanks to the definition of }\ \  u+m)\\
&= ((u\times m)_X\ts (u\ts m)_Y)\circ\La_{A\tx B,X,Y}
\end{split}
\end{equation*}
So $(A\tx B, u\times m)$ is an object of the category $C_{\A}$.

Assume that  $f: (A,u)\ri (B,m), g: (A',u')\ri (B',m')$ are two morphisms of $C_{\A}$. According to  [2], the morphism
\begin{equation*}
f\times g=f\otimes g: (A,u)\times (A',u')\ri (B,m)\times (B',m')
\end{equation*}
satisfies the condition  $(10)$, i.e.,  $f\times g$ is a morphism of  $C_{\A}$. According to  [2], $C_{\A}$ has an associativity constraint as follows:
\begin{equation*}
id: ((A,u)\times (B,m))\times (c,w)\ri (A,u)\times ((B,m)\times (C,w)).
\end{equation*}

We can easily varify that  $(I, id)$ is an object of  $C_{\A}$ and it together with the identity constraints $l=id, r=id$ is the unitivity constraint respect to  the operator  $\times$ of $C_{\A}$. Finally, according to [2], $C_{\A}$ is a braided tensor category with the  braiding $c$ given by: 
\begin{equation*}
c_{(A,u), (B,m)}=u_B: (A,u)\times (B,m)\ri (B,m)\times (A,u). 
\end{equation*}
\end{proof}
\begin{thm}
$C_{\A}$ is a braided Ann-category. 
\end{thm}
\begin{proof}
According to Proposition 7.1, $(C_{\A}, +)$ is a symmetric categorical group. According to  Proposition 7.2, $(C_{\A}, \times)$ is a braided tensor category. Moreover, we can verify that $C_{\A}$ has a distributivity constraints:
\begin{eqnarray*}
\La_{(A,u), (B,m), (C,w)}=\La_{A,B,C}, \  
\Ra_{(A,u), (B,m), (C,w)}=id.
\end{eqnarray*}
Since $\A$ is an Ann-category,  so is $C_{\A}$. On the other hand, each object $(A,u)$ of $C_{\A}$, natural isomorphism   $u$ satisfies the condition $(C3)$,  so the constraints $\La, \Ra=id, c$ of the category $C_{\A}$ satisfy the diagram (6). Since $id_0=id$,\linebreak $c_{(0,id), (0,id)}=id$. So $C_{\A}$ is a braided Ann-category.
\end{proof}

\vspace{0.2cm}
 
When  $\mathcal A$ is an almost strict Ann-category of the type $(R,M)$, we can describe more detail the center of $\A$. 

Let $\mathcal I $ be an almost strict Ann-category of the type $(R,M)$. Then, the center  $C_{\mathcal I}$ of  $\mathcal I$ is a category whose objects are pairs $(a,u)$, with $a\in R$, 
\[u(x): ax\ri xa\]
\noindent is a natural transformation satisfying the conditions (C1), (C2), (C3).

Firstly,  since $u(x): ax\ri xa$ is a morphism of a category of the type   $(R,M)$, so we have
\[ax=xa\ \ \forall x\in R\]
So $a\in Z(R)$.

Next, since $u$ satisfies  the condition (C1),  we have $u(1)=0$. The morphism   $u$ satisfies the condition   (C2), we have: 
\[u(xy)=xu(y)+yu(x).\]

Finally, $u$ satisfies the condition (C3), i.e., 
\[\lambda(a,x,y)=u(x)-u(x+y)+u(y)\]
So, each object of $\mathcal I$ is a pair $(a,u)$, in which  $a\in Z(R)$ and
\[u: R\ri M\]
is a function satisfying the three above conditions.

Let $f: (a,u)\ri (b,m)$ be a morphism of the category $C_{\mathcal I}$. Since $f: a\ri b$ is a morphism of  $\mathcal I$, we have $a=b$. The morphism $f$ satisfies the condition $(10)$, i.e., 
\[m(x)+(a,f)\otimes (0,x)=(0,x)\otimes (a,f)+u(x)\Leftrightarrow m(x)+xf=xf+u(x)\]

So $u=m$, i.e., morphisms of  $C_{\mathcal I}$ are all edomorphisms.

\begin{center}

\end{center}
\end{document}